\documentclass[a4paper]{amsart}

\usepackage[margin=3cm]{geometry}
\usepackage{palatino}

\usepackage{amsmath,amssymb,mathtools}
\usepackage{xcolor}
\usepackage[colorlinks=true,linkcolor=blue!55!black,citecolor=blue!55!black,
  urlcolor=blue!55!black]{hyperref}

\numberwithin{equation}{section}
\newtheorem{theorem}{Theorem}[section]
\newtheorem{proposition}[theorem]{Proposition}
\newtheorem{corollary}[theorem]{Corollary}

\newtheorem{definition}[theorem]{Definition}
\newtheorem{example}[theorem]{Example}
\theoremstyle{remark}
\newtheorem{remark}[theorem]{Remark}

\newcommand{\E}{\mathbb E}
\newcommand{\Pp}{\mathbb P}
\newcommand{\R}{\mathbb R}
\newcommand{\1}{\mathbf 1}
\newcommand{\Var}{\operatorname{Var}}

\title[Quadratic one-cut risk profiles]{One-Cut Risk Profiles under Quadratic Loss:\\
Discrete Convexity, Continuous Limits, and Higher Dimensions}

\author{Mihaela-Adriana Nistor}
\address{Faculty of Mathematics and Computer Science, University of Bucharest,
Bucharest, Romania}
\email{mihaelaadriana.nistor@gmail.com}

\author{Ionel Popescu}
\address{Faculty of Mathematics and Computer Science, University of Bucharest,
Bucharest, Romania; and Institute of Mathematics ``Simion Stoilow'' of the
Romanian Academy, Bucharest, Romania}
\email{ionel.popescu@fmi.unibuc.ro; ionel.popescu@imar.ro}

\date{September 4, 2026}


\hypersetup{
  pdftitle={One-Cut Risk Profiles under Quadratic Loss: Discrete Convexity, Continuous Limits, and Higher Dimensions},
  pdfauthor={Mihaela-Adriana Nistor and Ionel Popescu},
  pdfsubject={Quadratic two-regime risk profiles for scalar and vector laws},
  pdfkeywords={discrete log-concavity, weak symmetry, residual life, threshold optimization, risk measures}
}

\begin{document}
\raggedbottom

\begin{abstract}
In this note we study a two-regime representation of a loss random variable
under quadratic error.  For a finite law we compute exactly the change of the
optimal risk when one atom crosses the cut.  This turns the problem into a
convexity question in cumulative-mass coordinates.  On an equally spaced
support, log-concavity gives this convexity, while weak symmetry locates the
optimal cut, with an additional correction when the mean lies between two
atoms.  We also discuss the continuous analogue and extend the main identities
to finitely supported random vectors, where a global optimal partition may be
chosen as a halfspace.
\end{abstract}

\maketitle

\section{Introduction}\label{sec:introduction}

There is a large literature on monetary risk measures.  The coherent and
convex theories explain when a random loss can be summarized by one capital
requirement in a financially meaningful way; see
\cite{artzner-delbaen-eber-heath,frittelli-rosazza,foellmer-weber}.  Such a
number is useful, and often necessary, but it does not show whether the loss is
well represented by one typical level or whether two different regimes are
present.

We describe now the object studied in this paper.  Given a loss random
variable $X$, we approximate it by
\[
 r_{a,b,t}(X)=a\1_{\{X\leq t\}}+b\1_{\{X>t\}}
\]
and minimize
\begin{equation}\label{eq:basic-objective}
 \E\bigl[(X-r_{a,b,t}(X))^2\bigr]
\end{equation}
over $a,b$ and the cut point $t$.  For a fixed cut the best representatives are
the two conditional means.  The output consists of the partition, the two
representative levels and the error which remains after the split.  Since two
cuts may give the same value, the optimal representation is naturally
set-valued.  Notice also that the minimized error is not a cash-additive
monetary risk measure.  It is closer to a deviation functional and measures
the information lost when $X$ is described by two levels
\cite{rockafellar-uryasev-zabarankin-2006}.

The main point in the discrete case is that moving a cut across an atom is not
an infinitesimal operation.  The atom leaves one conditional sample and enters
the other one.  We show that the exact variation of the optimized risk is the
mass of the atom multiplied by the difference of two squared residual radii.
After division by the mass, this quantity is the slope of the profile in
cumulative-probability coordinates.  Thus the convexity of the whole profile
is equivalent to a simple monotonicity property of these slopes.

On an equally spaced support, log-concavity of the cumulative and survival
sequences implies the required monotonicity.  A log-concave mass sequence is a
convenient sufficient condition.  This should be compared with the continuous
case, where log-concavity and residual-life monotonicity play the same role
\cite{johnson-goldschmidt,bagnoli-bergstrom}.  Once convexity is available, a
weak symmetry assumption can be used to locate the minimum.  By weak symmetry
we mean only that the mean is also a median; no reflection symmetry of the law
is required.  If the mean is an atom, the two cuts around it are optimal.  If
the mean lies in a gap, an additional boundary-mass condition appears, and it
cannot in general be omitted.

The quadratic problem is also the ordered two-level quantization problem and,
in one dimension, it is equivalent to Otsu's between-class variance criterion
\cite{otsu,fisher-1958,lloyd-1982}.  Trushkin and Kieffer studied related
uniqueness questions for continuous quantization under shape assumptions
\cite{trushkin1982,kieffer1983}.  The purpose here is more specific: we isolate
the correction caused by atoms and the probability-coordinate convexity which
governs the optimal cut.

The atomless version of the problem was considered in
\cite{nistor-popescu-continuous}.  We include its derivative formula in order
to make the comparison with the discrete profile explicit.  We also record a
fine-grid compatibility result.  Finally, for finitely supported random
vectors, conditioning on parallel fibers produces the same marginal identity
for the fiber barycenters.  Nearest-centroid geometry then shows that a global
optimal two-regime partition may be chosen as a halfspace.  These properties
use the Pythagorean identity and are therefore particular to the quadratic
loss.

We finish the introduction with a description of the paper.  In
Section~\ref{sec:scalar-profile} we define the one-cut risk profile and prove
the exact atom update.  Section~\ref{sec:lattice} gives the lattice
log-concavity conditions and explains why the geometry of the support cannot
be ignored.  In Section~\ref{sec:weak-symmetry} we use weak symmetry to locate
the optimal cut and derive the correction needed when the mean lies between
two atoms.  Section~\ref{sec:continuous-benchmark} presents the continuous
counterpart and the discretization remarks.  The vector problem is treated in
Section~\ref{sec:vector-fixed-direction} for a fixed direction and in
Section~\ref{sec:global-halfspace} for arbitrary two-regime partitions.
Section~\ref{sec:vector-examples} contains two examples which show the limits
of the symmetry and direction assumptions.  The last section summarizes the
main conclusions.

\part{Scalar laws: discrete convexity and continuous limits}

\section{The scalar quadratic risk representation}
\label{sec:scalar-profile}

In this section we introduce the scalar profile and compute its variation when
one atom is moved from the right regime to the left one.  This simple identity
will be the main tool in the discrete part of the paper.

Throughout the scalar part, larger values of $X$ are interpreted as larger
losses.  Let $X\in L^2$ be nonconstant.  For a threshold $t$ with two
nonempty regimes, set
\begin{equation}\label{eq:profiled-risk-general}
 \begin{split}
 J_X(a,b;t)&:=\E\bigl[(X-a)^2\1_{\{X\leq t\}}
                    +(X-b)^2\1_{\{X>t\}}\bigr],\\
 W_X(t)&:=\inf_{a,b\in\R}J_X(a,b;t).
 \end{split}
\end{equation}
For a fixed cut the minimizers are unique and are given by
\begin{equation}\label{eq:conditional-regime-levels}
 a_X(t)=\E[X\mid X\leq t],\qquad
 b_X(t)=\E[X\mid X>t].
\end{equation}
We will call $t\mapsto W_X(t)$ the \emph{one-cut quadratic risk profile}.  Let
\[
 \mathcal T_X:=\{t\in\R:0<\mathbb P(X\leq t)<1\}
\]
be its admissible threshold set, and write $\mu_X:=\mathcal L(X)$.  We identify
$s,t\in\mathcal T_X$ when
$\mu_X((\min\{s,t\},\max\{s,t\}])=0$ and denote the resulting class by
$[t]_{\mu_X}$.  The best two-regime risk representations are the parameter
triples
\begin{equation}\label{eq:set-valued-representation}
 \mathfrak R_2(\mu_X):=
 \left\{
 \bigl(a_X(t),b_X(t),[t]_{\mu_X}\bigr):
 t\in\mathcal T_X,\quad W_X(t)=\inf_{s\in\mathcal T_X}W_X(s)
 \right\},
\end{equation}
whenever the infimum is attained.  A triple $(a,b,[t]_{\mu_X})$ gives the
random variable $a\1_{\{X\leq t\}}+b\1_{\{X>t\}}$ on any probability space
carrying $X$.  The equivalence class removes the irrelevant choice of a point
inside a gap of the support.  In particular,
\eqref{eq:set-valued-representation} depends only on the law of $X$.

We record first two immediate properties.  A direct change of variables gives
\[
 W_{X+c}(t+c)=W_X(t),\qquad
 W_{\lambda X}(\lambda t)=\lambda^2W_X(t)
 \quad(c\in\R,\ \lambda>0),
\]
and an optimal triple $(a,b,[t])$ is carried to
$(a+c,b+c,[t+c])$ or $(\lambda a,\lambda b,[\lambda t])$, respectively.

We now specialize to a finite law.  In this case an optimal partition always
exists and the entire profile can be obtained from finitely many adjacent
moves.  Let
\begin{equation}\label{eq:finite-law}
  x_1<\cdots<x_n,\qquad
  p_i=\Pp(X=x_i)>0,\qquad \sum_{i=1}^n p_i=1,
\end{equation}
with $n\geq2$.  Put
\[
  P_i:=\sum_{j\leq i}p_j,\qquad
  T_i:=\sum_{j\geq i}p_j=1-P_{i-1},
  \qquad P_0:=0,
\]
and write $\mu=\E X$.  For $1\leq i<n$, the optimized squared risk at
the cut after $x_i$ is
\begin{equation}\label{eq:scalar-risk}
  W_i:=\min_{a,b\in\R}
  \left\{\sum_{j\leq i}p_j(x_j-a)^2
        +\sum_{j>i}p_j(x_j-b)^2\right\}.
\end{equation}
The minimizers are the conditional means
\[
  a_i:=\E[X\mid X\leq x_i],\qquad
  b_i:=\E[X\mid X>x_i].
\]
For later use, include the one-regime endpoints
\[
  W_0=W_n:=\Var(X),\qquad B_i:=\Var(X)-W_i,
  \qquad B_0=B_n:=0.
\]
The usual within/between decomposition gives
\begin{equation}\label{eq:gain-formula}
 B_i=P_i(1-P_i)(a_i-b_i)^2
 =\frac{\bigl(\E[(X-\mu)\1_{\{X\leq x_i\}}]\bigr)^2}
        {P_i(1-P_i)},\qquad 1\leq i<n.
\end{equation}

The next quantities contain the correction created by the atom.  Notice that
both the strict and the non-strict conditional laws enter their definition.
For $1\leq i<n$ and $1<i\leq n$, respectively, define
\begin{align}
 A_i&:=\sum_{j>i}p_j(x_j-x_i),
 &R_i&:=\frac{A_i}{\sqrt{T_iT_{i+1}}},\label{eq:right-radius}\\
 C_i&:=\sum_{j<i}p_j(x_i-x_j),
 &L_i&:=\frac{C_i}{\sqrt{P_{i-1}P_i}},\label{eq:left-radius}
\end{align}
with $L_1:=0$ and $R_n:=0$.  Equivalently, \eqref{eq:right-product} holds for
$i<n$, while \eqref{eq:left-product} holds for $i>1$:
\begin{align}
 R_i^2
 &=\E[X-x_i\mid X>x_i]\,
   \E[X-x_i\mid X\geq x_i],\label{eq:right-product}\\
 L_i^2
 &=\E[x_i-X\mid X<x_i]\,
   \E[x_i-X\mid X\leq x_i].\label{eq:left-product}
\end{align}
Thus $R_i$ and $L_i$ are geometric residual radii.  They are not the usual
arithmetic mean residual lives.

\begin{proposition}[Exact atom update]\label{prop:atom-update}
For every $1\leq i\leq n$,
\begin{equation}\label{eq:atom-update}
 W_i-W_{i-1}=p_i(L_i^2-R_i^2),\qquad
 B_i-B_{i-1}=p_i(R_i^2-L_i^2).
\end{equation}
Consequently, moving $x_i$ from the right regime to the left weakly lowers
the optimized risk exactly when $R_i\geq L_i$.  At $i=1$ and $i=n$ the
identity compares a two-regime cut with the one-regime endpoint; only the
interior indices compare two nonempty adjacent splits.
\end{proposition}

\begin{proof}
We prove first the interior identities.  For $2\leq i\leq n-1$, adding
$p_i\delta_{x_i}$ to the left block changes its optimal sum of squares by
\[
 p_i(x_i-a_{i-1})(x_i-a_i)=p_iL_i^2.
\]
At the same time, removing this atom from the right block decreases the cost by
\[
 p_i(b_{i-1}-x_i)(b_i-x_i)=p_iR_i^2.
\]
These are the standard weighted variance identities.  Taking their difference
gives \eqref{eq:atom-update} for the interior indices.  It remains to consider
the two endpoints.  In this case \eqref{eq:gain-formula} gives directly
\[
  B_1=p_1R_1^2,\qquad B_{n-1}=p_nL_n^2.
\]
Since $B_0=B_n=0$ and $L_1=R_n=0$, the proof is complete.
\end{proof}

Notice the factor $p_i$ in \eqref{eq:atom-update}.  It explains why the
sequence $(W_i)$ need not be convex in the index $i$.  The correct variable is
the cumulative mass.

\begin{corollary}[Convexity in probability coordinates]
\label{cor:mass-polygon}
Let $\mathcal W$ be the polygonal interpolation through
\[
  (P_i,W_i),\qquad 0\leq i\leq n.
\]
Its slope on $[P_{i-1},P_i]$ is
\begin{equation}\label{eq:polygon-slope}
  \frac{W_i-W_{i-1}}{p_i}=L_i^2-R_i^2.
\end{equation}
Hence $\mathcal W$ is convex if and only if
\begin{equation}\label{eq:signal-monotonicity}
  D_i:=R_i^2-L_i^2
  \quad\hbox{is nonincreasing in }i.
\end{equation}
Under \eqref{eq:signal-monotonicity}, the minimizing cut indices form one
consecutive block.  If $(D_i)$ is strictly decreasing, there is one minimizing
vertex, except that $D_i=0$ produces exactly the two endpoints of one flat
edge.
\end{corollary}

The corollary gives a convenient reformulation of the optimization problem:
the whole profile is convex precisely when its consecutive slopes increase in
cumulative-mass coordinates.  The separate monotonicity conditions
\begin{equation}\label{eq:separate-radii}
  R_1\geq\cdots\geq R_n,
  \qquad L_1\leq\cdots\leq L_n
\end{equation}
are an immediate sufficient condition for \eqref{eq:signal-monotonicity}, but
they are not necessary.

\section{Lattice log-concavity and arbitrary supports}
\label{sec:lattice}

We now give a simple condition which implies the monotonicity of the two
residual radii.  The equal spacing of the support is important here.  Recall
that a positive sequence
$(s_i)$ is log-concave if $s_i^2\geq s_{i-1}s_{i+1}$ whenever the three terms
occur.

\begin{theorem}[Cumulative log-concavity on a lattice]
\label{thm:cumulative-lc}
Suppose that
\[
  x_i=x_1+(i-1)h,\qquad h>0,
\]
and that both $(P_i)$ and $(T_i)$ are log-concave.  Then $(R_i)$ is strictly
decreasing and $(L_i)$ is strictly increasing.  In particular, the
probability-coordinate risk profile has strictly increasing consecutive
slopes.  The optimized cut is
unique, apart from a possible tie between two adjacent cuts.
\end{theorem}

\begin{proof}
Put $\tau_i=T_{i+1}/T_i$.  The log-concavity of $T$ says exactly that
$(\tau_i)$ is nonincreasing.  On the other hand, the tail-sum identity gives
\[
 A_i=h\sum_{m=i+1}^nT_m
 \,\quad\hbox{and hence}\quad
 R_i=h\sqrt{\tau_i}
 \bigl(1+\tau_{i+1}+\tau_{i+1}\tau_{i+2}+\cdots\bigr).
\]
Both factors in the last expression are nonincreasing.  More precisely,
\[
 \frac{A_i}{T_{i+1}}
 =h\bigl(1+\tau_{i+1}+\tau_{i+1}\tau_{i+2}+\cdots\bigr)
 \geq \frac{A_{i+1}}{T_{i+2}}
\]
by termwise comparison.  The last nonzero term makes the inequality between
the corresponding radii strict.  Applying the same argument to the reflected
distribution and using the log-concavity of $P$ proves that $L_i$ is strictly
increasing.
\end{proof}

\begin{corollary}[Log-concave masses]\label{cor:pmf-lc}
If the mass sequence satisfies
\begin{equation}\label{eq:pmf-lc}
  p_i^2\geq p_{i-1}p_{i+1},\qquad 1<i<n,
\end{equation}
then the conclusion of Theorem~\ref{thm:cumulative-lc} holds.
\end{corollary}

\begin{proof}
The partial sums and the tails are convolutions of $(p_i)$ with a one-sided
all-ones sequence.  Their log-concavity therefore follows from the
preservation of log-concavity under convolution
\cite{johnson-goldschmidt}.  For completeness, the tail assertion can also be
seen directly.  The ratios $p_{i+1}/p_i$ decrease, and the shifted-product
expansion of $T_i/p_i$ shows that $T_{i+1}/T_i$ decreases.  The cumulative
assertion follows by reflection.
\end{proof}

The hypothesis on the cumulative sequences is weaker than
\eqref{eq:pmf-lc}.  For instance, masses proportional to $(1,1,2,1)$ have
log-concave cumulative and survival sequences.  On the other hand, equal
spacing cannot be dropped.  Log-concavity of the masses does not control an
arbitrary support.  To see this, take the uniform law on
\[
  \{-20,-19,-2,2,19,20\}
\]
has a log-concave mass array and is reflection symmetric, but its gains are
\[
  80,\quad \frac{1521}{8},\quad \frac{1681}{9},\quad
  \frac{1521}{8},\quad 80.
\]
Thus the two off-center cuts are optimal, while the central cut is not.

\section{Weak symmetry locates the crossing}
\label{sec:weak-symmetry}

The monotonicity from the preceding sections determines the shape of the
profile.  We use now a centering condition to locate its minimum.  Let
$\mu=\E X$ and assume the weak balance
\begin{equation}\label{eq:weak-balance}
  \Pp(X<\mu)=\Pp(X>\mu).
\end{equation}
Notice that no reflection property of the distribution is assumed.

\begin{corollary}[Mean-atom balance]\label{cor:mean-atom}
Suppose that $\mu=x_k$ and
\[
  \Pp(X<\mu)=\Pp(X>\mu)=q>0.
\]
Then the cuts immediately below and immediately above $\mu$ have the same
gain.  If \eqref{eq:signal-monotonicity} holds, both are globally optimal;
if the signal is strictly decreasing, these are the only optimal cuts.
\end{corollary}

\begin{proof}
Set $M=\E[(X-\mu)\1_{\{X>\mu\}}]$.  Since $\E(X-\mu)=0$, the corresponding
left moment is $-M$.  Formula \eqref{eq:gain-formula} shows that both cuts have
gain $M^2/[q(1-q)]$.  Hence $B_{k-1}=B_k$, or equivalently
$W_{k-1}=W_k$.  The polygonal profile therefore has slope zero on the interval
$[P_{k-1},P_k]$.  Under \eqref{eq:signal-monotonicity} the profile is convex,
so its slopes are nondecreasing.  All slopes before this interval are thus
nonpositive and all slopes after it are nonnegative.  Consequently the profile
decreases up to this flat edge and increases after it, which proves that both
cuts are globally optimal.  If $D_i$ is strictly decreasing, the slopes are
strictly negative before this edge and strictly positive after it; hence the
two cuts are the only minimizers.
\end{proof}

The condition is genuinely weaker than symmetry of the law.  For example,
the masses proportional to $(5,6,7,7,3,1)$ on
$\{-2,-1,0,1,2,3\}$ are strictly log-concave, have mean $0$ and balanced
tails, but are not reflection symmetric; the two cuts bracketing $0$ are the
only optimizers.

We turn now to the case where the mean is not an atom.  Here discreteness
produces an additional condition.  Suppose
\begin{equation}\label{eq:offgrid-setup}
 x_k<\mu<x_{k+1},\qquad
 \Pp(X<\mu)=\Pp(X>\mu)=\frac12.
\end{equation}
Write
\begin{align*}
 d&:=\E[X-\mu\mid X>\mu],\\
  &=\E[\mu-X\mid X<\mu],\\
 z_-&:=\mu-x_k,\quad z_+:=x_{k+1}-\mu,\\
 p_-&:=p_k,\qquad p_+:=p_{k+1}.
\end{align*}
The gain of the mean-gap cut is $B_k=d^2$.

\begin{theorem}[Sharp off-grid atom correction]
\label{thm:offgrid-correction}
Assume \eqref{eq:offgrid-setup} and that the signal $D_i$ is nonincreasing.
Then the mean-gap cut is globally optimal if and only if
\begin{equation}\label{eq:boundary-atom-condition}
 p_-\leq\frac{dz_-}{d^2+z_-^2},\qquad
 p_+\leq\frac{dz_+}{d^2+z_+^2},
\end{equation}
with an inequality omitted if the neighbouring two-regime cut is
inadmissible.  Strict inequalities together with a strictly decreasing
signal give a unique optimum.  Equality on one side makes the adjacent cut
optimal as well; under a merely nonincreasing signal, a larger consecutive
tie block is possible.
\end{theorem}

\begin{proof}
We write the gain in terms of the centered left moment
\[
 S_i:=\E[(X-\mu)\1_{\{X\leq x_i\}}],
 \qquad
 B_i=\frac{S_i^2}{P_i(1-P_i)}.
\]
For the mean-gap cut, $P_k=1/2$ and
\[
 S_k=\E[(X-\mu)\1_{\{X<\mu\}}]=-\frac{d}{2},
\]
so indeed $B_k=d^2$.

Suppose first that the cut after $x_{k+1}$ is admissible.  Moving the atom at
$x_{k+1}=\mu+z_+$ to the left changes the left mass and its centered moment to
\[
 P_{k+1}=\frac12+p_+,\qquad
 S_{k+1}=-\frac d2+p_+z_+.
\]
Consequently,
\[
 B_{k+1}
 =\frac{(d/2-p_+z_+)^2}{(1/2+p_+)(1/2-p_+)}
 =\frac{(d/2-p_+z_+)^2}{1/4-p_+^2}.
\]
Subtracting $B_k=d^2$ and simplifying gives
\begin{align}
 B_{k+1}-B_k
 &=\frac{p_+\{p_+(d^2+z_+^2)-dz_+\}}
         {1/4-p_+^2}.\label{eq:right-atom-correction}
\end{align}

The calculation on the left is the same.  If the cut after $x_{k-1}$ is
admissible, removing the atom at $x_k=\mu-z_-$ from the left gives
\[
 P_{k-1}=\frac12-p_-,\qquad
 S_{k-1}=-\frac d2+p_-z_-,
\]
and therefore
\begin{align}
 B_{k-1}-B_k
 &=\frac{p_-\{p_-(d^2+z_-^2)-dz_-\}}
         {1/4-p_-^2}.\label{eq:left-atom-correction}
\end{align}
The denominators are positive because $p_\pm<1/2$.  Hence
$B_{k+1}\leq B_k$ is equivalent to the right inequality in
\eqref{eq:boundary-atom-condition}, while $B_{k-1}\leq B_k$ is equivalent to
the left one.  Thus \eqref{eq:boundary-atom-condition} is exactly the local
maximum condition for the gain.

It remains to explain why this local condition is global.  By
\eqref{eq:atom-update},
\[
 B_i-B_{i-1}=p_iD_i.
\]
The inequality $B_{k-1}\leq B_k$ gives $D_k\geq0$, and
$B_{k+1}\leq B_k$ gives $D_{k+1}\leq0$.  Since $(D_i)$ is nonincreasing, we
have $D_i\geq0$ for every $i\leq k$ and $D_i\leq0$ for every $i\geq k+1$.
Thus $(B_i)$ is nondecreasing up to $k$ and nonincreasing after $k$, so $B_k$
is a global maximum, equivalently $W_k$ is a global minimum.  The same
argument, with only one side, applies when a neighboring two-regime cut is
inadmissible.

If both inequalities in \eqref{eq:boundary-atom-condition} are strict, then
$D_k>0>D_{k+1}$.  Under strict decrease of the signal, all gain increments
before $k$ are strictly positive and all those after $k$ are strictly
negative, which proves uniqueness.  Equality on one side gives a zero
increment and hence a tie with the corresponding adjacent cut.  If the signal
is only nonincreasing, several consecutive zero increments may produce a
larger tie block.
\end{proof}

On a lattice this correction has a particularly simple form.

\begin{corollary}[Midpoint weak symmetry]\label{cor:midpoint-symmetry}
Assume the hypotheses of Theorem~\ref{thm:cumulative-lc}, that $\mu$ is the
midpoint of the adjacent sites $x_k,x_{k+1}$, and
\eqref{eq:offgrid-setup}.  Then the mean-gap cut is the unique optimal cut.
\end{corollary}

\begin{proof}
We prove first the right inequality in
\eqref{eq:boundary-atom-condition}.  Since $\mu$ is the midpoint of
$x_k$ and $x_{k+1}$, we have $z_+=h/2$.  Conditional on $X>\mu$, the random
variable
\[
 J:=\frac{X-x_{k+1}}{h}
\]
takes values in $\{0,1,2,\ldots\}$.  Put
\[
 q_0:=\Pp(J=0\mid X>\mu),
 \qquad S_j:=\Pp(J\geq j\mid X>\mu),\quad j\geq0.
\]
Because $\Pp(X>\mu)=1/2$, the mass at the first point of the right half is
\[
 p_+=\Pp(X=x_{k+1})=\frac{q_0}{2}.
\]

The sequence $(S_j)$ is a normalized subsequence of the survival sequence
$(T_i)$ and is therefore log-concave.  Thus the ratios
$S_{j+1}/S_j$ are nonincreasing.  Since $S_0=1$ and
$S_1=1-q_0$, we obtain
\[
 \frac{S_{j+1}}{S_j}\leq\frac{S_1}{S_0}=1-q_0,
 \qquad j\geq0.
\]
Iterating this inequality gives
\[
 S_j\leq(1-q_0)^j,
\]
and the tail-sum formula for the nonnegative integer-valued random variable
$J$ yields
\[
 \E[J\mid X>\mu]=\sum_{j\geq1}S_j
 \leq\sum_{j\geq1}(1-q_0)^j
 =\frac{1-q_0}{q_0}.
\]
Here and below the conditional expectation notation is suppressed because
$J$ has been defined under the conditional law $X>\mu$.

Moreover,
\[
 X-\mu=(X-x_{k+1})+(x_{k+1}-\mu)
       =h\left(J+\frac12\right)
 \quad\hbox{on }\{X>\mu\}.
\]
Consequently,
\[
 d=h\left(\E J+\frac12\right),
 \qquad
 d+\frac h2=h(\E J+1)\leq\frac{h}{q_0}.
\]
It follows that
\[
 p_+=\frac{q_0}{2}
 \leq\frac{h}{2(d+h/2)}.
\]
Finally $J\geq0$, and hence $d\geq h/2$.  This last inequality is equivalent
to
\[
 \frac{h}{2(d+h/2)}
 \leq\frac{d(h/2)}{d^2+h^2/4}
 =\frac{dz_+}{d^2+z_+^2}.
\]
We have proved the required bound for $p_+$.  Applying the same argument to
$-X$ gives the bound for $p_-$.

If the neighboring right cut is admissible, then there is positive mass
strictly above $x_{k+1}$.  Since the support is finite, the geometric tail
bound is then strict at some index, and so the inequality for $p_+$ is strict.
The same observation applies on the left.  Finally,
Theorem~\ref{thm:cumulative-lc} gives a strictly decreasing signal $D_i$.
The uniqueness assertion now follows from
Theorem~\ref{thm:offgrid-correction}.
\end{proof}

We point out that the midpoint assumption cannot simply be omitted.  Consider
the support $\{-1,0,1,2,3\}$ and take masses
\[
  \frac1{26}(1,12,7,4,2).
\]
They are strictly log-concave, $\mu=10/13\in(0,1)$, and each side of the mean
has probability $1/2$.  Nevertheless, the mean-gap gain is $121/169$, while
the gain after $1$ is $3721/5070>121/169$.  Here
\[
 d=\frac{11}{13},\qquad z_+=\frac3{13},\qquad p_+=\frac7{26},
\]
and the right condition in \eqref{eq:boundary-atom-condition} fails, since
$7/26>33/130$.

\section{The continuous benchmark and discrete limits}
\label{sec:continuous-benchmark}

We discuss here the atomless version of the preceding formulas.  The natural
horizontal variable is again probability mass and not the numerical value of
the threshold.  This distinction is useful when the support has gaps, because
the threshold profile is constant on a gap while the quantile profile has no
such redundant parameter.

Let $X\in L^2$ have an atomless, nondegenerate distribution function $F$, and
let $Q(u)=\inf\{x:F(x)\geq u\}$ be its quantile function.  For $0<u<1$ define
\begin{align}
 a(u)&:=\frac1u\int_0^u Q(v)\,dv,\\
 b(u)&:=\frac1{1-u}\int_u^1Q(v)\,dv,
 \label{eq:continuous-centres}\\
 \mathcal W_X(u)&:=
 \int_0^u\bigl(Q(v)-a(u)\bigr)^2\,dv
 +\int_u^1\bigl(Q(v)-b(u)\bigr)^2\,dv.
 \label{eq:continuous-quantile-profile}
\end{align}
At $u=0,1$ we set $\mathcal W_X(u)=\Var(X)$.  If $u=F(t)$ and $t$ belongs to
the support, \eqref{eq:continuous-quantile-profile} is equal to $W_X(t)$ from
\eqref{eq:profiled-risk-general}.  The next result is therefore the continuous
counterpart of Corollary~\ref{cor:mass-polygon}.

\begin{theorem}[Continuous residual signal]
\label{thm:continuous-residual-signal}
Define, for every $t$ in $J(F)=\{t:0<F(t)<1\}$,
\begin{equation}\label{eq:continuous-residual-radii}
 \ell(t):=\E[t-X\mid X<t],\qquad
 r(t):=\E[X-t\mid X>t],\qquad D(t):=r(t)^2-\ell(t)^2.
\end{equation}
The profile $\mathcal W_X$ is locally absolutely continuous on $(0,1)$ and,
for almost every $u$,
\begin{equation}\label{eq:continuous-profile-derivative}
 \mathcal W_X'(u)
 =\bigl(Q(u)-a(u)\bigr)^2-\bigl(b(u)-Q(u)\bigr)^2
 =\ell(Q(u))^2-r(Q(u))^2.
\end{equation}
Consequently, $\mathcal W_X$ is convex if and only if
$u\mapsto D(Q(u))$ has a nonincreasing representative.  If both $F$ and
$1-F$ are log-concave on $J(F)$, then $\ell$ is nondecreasing, $r$ is
nonincreasing, and this convexity condition holds.

If $D$ is nonincreasing and the mean $\mu$ is also a median, then $u=1/2$
minimizes $\mathcal W_X$.  It is the unique minimizer whenever $D$ is strictly
decreasing on the support.
\end{theorem}

\begin{proof}
We start with the derivative.  Minimizing each integral in
\eqref{eq:continuous-quantile-profile}, we have for example
\[
 \min_c\int_0^u(Q(v)-c)^2\,dv
 =\int_0^uQ(v)^2\,dv-\frac1u\left(\int_0^uQ(v)\,dv\right)^2.
\]
These expressions are absolutely continuous on every compact subinterval of
$(0,1)$.  Differentiation at a Lebesgue point of $Q$ gives
$(Q(u)-a(u))^2$ for the left term and $-(b(u)-Q(u))^2$ for the right term.
This proves \eqref{eq:continuous-profile-derivative}.  Now a locally
absolutely continuous function is convex if and only if its derivative has a
nondecreasing representative.  This gives the equivalence with the signal
condition.

In the second place, write $S=1-F$.  The usual tail formulas give
\[
 r(t)=\int_0^\infty\frac{S(t+s)}{S(t)}\,ds,
 \qquad
 \ell(t)=\int_0^\infty\frac{F(t-s)}{F(t)}\,ds,
\]
where the integrands are set equal to zero outside the support.  If
$\phi=\log S$ is concave, then for every $s>0$ the increment
$\phi(t+s)-\phi(t)$ is nonincreasing in $t$; hence $r$ is nonincreasing.
If $\psi=\log F$ is concave, then
$\psi(t-s)-\psi(t)$ is nondecreasing in $t$; hence $\ell$ is
nondecreasing.  It follows that $D$ is nonincreasing and the profile is
convex.

Finally, suppose that $F(\mu)=1/2$ and $\E X=\mu$.  Then
\[
 \E[(X-\mu)_+]=\E[(\mu-X)_+],
\]
and consequently $r(\mu)=\ell(\mu)$.  The convexity of the profile places its
minimum at $u=1/2$.  If $D$ is strictly decreasing, no other minimizer is
possible.
\end{proof}

Every nondegenerate log-concave law on the line has a log-concave density.
Moreover, $F$ and $1-F$ are log-concave by marginalization
\cite{prekopa1973,bagnoli-bergstrom,saumard-wellner}.  The theorem therefore
recovers the convexity mechanism used in the continuous model
\cite{nistor-popescu-continuous}.  At the same time, it shows that
bi-log-concavity is already sufficient \cite{duembgen-kolesnyk-wilke}.  The
exact condition remains the monotonicity of $D$.

\begin{remark}[Fine-grid compatibility]
\label{rem:fine-grid-compatibility}
For completeness, we record the relation with fine discretizations.  Call a
measure $\nu$ log-concave when it satisfies the Pr\'ekopa inequality
$\nu((1-\lambda)A+\lambda B)\geq
\nu(A)^{1-\lambda}\nu(B)^\lambda$ for compact sets $A,B$ and
$0<\lambda<1$.  For $h>0$, $a\in\R$, and $\nu=\mathcal L(X)$, let
\[
 I_{h,a,k}:=(a+(k-1)h,a+kh],\qquad
 p_{h,a,k}:=\nu(I_{h,a,k}).
\]
If $\nu$ is log-concave, then $(p_{h,a,k})_{k\in\mathbb Z}$ is log-concave
and its positive support is an interval.  Conversely, suppose that for some
$h_n\downarrow0$ and origins $a_n$ the positive cell masses have interval
support and
\[
 p_{h_n,a_n,k}^{2}\geq
 p_{h_n,a_n,k-1}p_{h_n,a_n,k+1}
 \qquad(k\in\mathbb Z).
\]
Then $\nu$ is log-concave; in one dimension it is either a point mass or has a
log-concave density.  The forward implication follows from Pr\'ekopa's theorem.
Indeed, convolution with a uniform
interval shows that $t\mapsto\nu((t-h,t])$ is log-concave, and restriction to
the grid gives the sequence inequality \cite{prekopa1973}.

We sketch the converse.  If a maximal cell mass does not tend to zero, an atom
appears in the limit.  The concavity of the logarithms of the masses, together
with interval support, then shows that no positive mass can stay a fixed
distance from this atom; otherwise the sum over the order $1/h_n$ intervening
cells would be larger than one.  Thus the limit is a point mass.  In the
remaining case let $M_n:=\max_k p_{h_n,a_n,k}\to0$, set
$c_{n,k}:=a_n+(k-\tfrac12)h_n$, let $\widetilde g_n$ be the log-linear
interpolant of $p_{h_n,a_n,k}/h_n$ at the points $c_{n,k}$, and put
$Z_n=\int\widetilde g_n$.  Also define the centered step density
\[
 s_n:=\sum_k\frac{p_{h_n,a_n,k}}{h_n}
       \1_{[c_{n,k},c_{n,k+1})}.
\]
Logarithmic-mean bounds and unimodality give
\[
 1-M_n\leq Z_n\leq1,
 \qquad \|\widetilde g_n-s_n\|_1\leq3M_n.
\]
Since $|1-Z_n|\leq M_n$, it follows that
\[
 \|Z_n^{-1}\widetilde g_n-s_n\|_1\leq4M_n.
\]
The law with density $s_n$ can be coupled with $\nu$ with displacement at most
$3h_n/2$.  Therefore the normalized log-concave interpolants converge weakly
to $\nu$.  The weak closure of the class of log-concave measures gives the
claim \cite{saumard-wellner}.  We stress that this compatibility result is not
used in the optimization arguments.
\end{remark}

We end this section with an exact comparison of the risk profiles.  Let
$\mathcal P_h$ consist of
intervals $I_k=(t_{k-1},t_k]$ of length at most $h$, index its nonempty cells
in increasing order, and set $X_h:=\E[X\mid\sigma(\mathcal P_h)]$.  At every
cell boundary $t_k$ with two nonempty sides, the event $\{X_h\leq t_k\}$ is
$\{X\leq t_k\}$, and conditional Pythagoras gives
\begin{equation}\label{eq:centroidal-risk-identity}
 W_X(t_k)=\E[(X-X_h)^2]+W_{X_h}(t_k).
\end{equation}
Indeed, the cross term between $X-X_h$ and any two-level function of $X_h$ has
conditional mean zero.  The first term in
\eqref{eq:centroidal-risk-identity} does not depend on the cut and is bounded
by $h^2/4$.  Thus centroidal discretization recovers the quadratic profile at
the cell boundaries as the mesh tends to zero.  In contrast with replacement
by cell centers, the decomposition is exact, even though the centroids need
not be equally spaced.

\part{Quadratic risk profiles in several dimensions}

The scalar argument uses the natural order of the real line.  In several
dimensions an order appears only after a direction has been fixed, and points
with the same projection form a fiber rather than a single atom.  The
Pythagorean identity separates the dispersion inside each fiber from the
motion of the fiber barycenters.  We first study this ordered problem and then
return to arbitrary two-regime partitions.

\section{Quadratic vector risk along a fixed direction}
\label{sec:vector-fixed-direction}

We formulate here the fixed-direction problem.  The main observation is that
the within-fiber error is independent of the cut, so the scalar calculation
can be repeated for the ordered barycenters.

Let $X$ be a finitely supported random vector in $\mathbb R^d$, let
$\mu:=\mathbb E X$, and fix a symmetric positive definite matrix $M$.
We write $\langle x,y\rangle_M:=x^{\mathsf T}My$ and
$\|x\|_M^2:=\langle x,x\rangle_M$.  Fix $u\neq0$ and order the distinct
values of
\[
  Y:=u^{\mathsf T}(X-\mu)
\]
as $t_1<\cdots<t_m$.  The probability and the barycenter of the $i$th
fiber are
\[
  p_i:=\mathbb P(Y=t_i),
  \qquad
  z_i:=\mathbb E[X-\mu\mid Y=t_i].
\]
Set $P_i:=\sum_{j\leq i}p_j$ and $q_i:=1-P_i$.  For the nonempty sides
define
\[
  a_i:=\frac1{P_i}\sum_{j\leq i}p_jz_j\quad(1\leq i\leq m),
  \qquad
  b_i:=\frac1{q_i}\sum_{j>i}p_jz_j\quad(0\leq i\leq m-1).
\]
Notice that $a_m=b_0=0$.

Let $V_i$ be the least quadratic risk obtained from the partition
$\{Y\leq t_i\}\mid\{Y>t_i\}$, and put
\[
  V_0=V_m=:T,
  \qquad
  T:=\mathbb E\|X-\mu\|_M^2.
\]
At $i=0,m$ the empty representative is immaterial and the nonempty
representative is $\mu$.

\begin{theorem}[Fixed-direction fiber reduction]
\label{thm:fiber-reduction}
Let $z(Y):=z_i$ on $\{Y=t_i\}$ and
\[
  \sigma^2_{\mathrm{fib}}
  :=\mathbb E\|X-\mu-z(Y)\|_M^2.
\]
Then, for $1\leq i\leq m-1$,
\begin{equation}\label{eq:fiber-decomposition}
  V_i=\sigma^2_{\mathrm{fib}}
      +\sum_{j\leq i}p_j\|z_j-a_i\|_M^2
      +\sum_{j>i}p_j\|z_j-b_i\|_M^2.
\end{equation}
At the endpoints,
\[
 V_0=V_m=T=\sigma^2_{\mathrm{fib}}+\sum_{j=1}^m p_j\|z_j\|_M^2.
\]
Thus the dispersion inside the fibers contributes the same constant to every
cut.  For $1\leq i\leq m-1$, put
\[
  S_i:=\mathbb E[(X-\mu)\mathbf 1_{\{Y\leq t_i\}}]
      =P_i a_i=-q_i b_i,
\]
then
\begin{equation}\label{eq:vector-gain}
  V_i=T-B_i,
  \qquad
  B_i=P_iq_i\|a_i-b_i\|_M^2
     =\frac{\|S_i\|_M^2}{P_iq_i},
\end{equation}
with $B_0=B_m=0$.
\end{theorem}

\begin{proof}
Set $\varepsilon:=X-\mu-z(Y)$.  We have
$\mathbb E[\varepsilon\mid Y]=0$.  If $c$ is any vector-valued function of
$Y$, then
\[
 \mathbb E\langle\varepsilon,c(Y)\rangle_M
 =\mathbb E\left[
   \left\langle\mathbb E[\varepsilon\mid Y],c(Y)\right\rangle_M
  \right]=0.
\]
Thus the part of $X-\mu$ orthogonal to the fibers has zero cross term with
every quantity which depends only on the fiber.

For the cut after $t_i$, the conditional mean on the left is
\[
 \mathbb E[X\mid Y\leq t_i]
 =\mu+\mathbb E[z(Y)\mid Y\leq t_i]
 =\mu+a_i,
\]
and similarly the conditional mean on the right is $\mu+b_i$.  On the left
side we can therefore write
\[
 X-(\mu+a_i)=\varepsilon+z(Y)-a_i.
\]
Expanding the squared $M$--norm and conditioning on $Y$ gives
\begin{align*}
 &\mathbb E\left[
   \|X-(\mu+a_i)\|_M^2\mathbf 1_{\{Y\leq t_i\}}
  \right]\\
 &\qquad=
 \mathbb E\left[\|\varepsilon\|_M^2
                 \mathbf 1_{\{Y\leq t_i\}}\right]
 +\sum_{j\leq i}p_j\|z_j-a_i\|_M^2,
\end{align*}
because the cross term vanishes.  The same calculation on the right yields
\[
 \mathbb E\left[
   \|X-(\mu+b_i)\|_M^2\mathbf 1_{\{Y>t_i\}}
  \right]
 =\mathbb E\left[\|\varepsilon\|_M^2
                 \mathbf 1_{\{Y>t_i\}}\right]
 +\sum_{j>i}p_j\|z_j-b_i\|_M^2.
\]
Adding the two identities proves \eqref{eq:fiber-decomposition}.  Taking one
regime only, with representative $\mu$, gives
\[
 T=\sigma^2_{\mathrm{fib}}+sum_{j=1}^m p_j\|z_j\|_M^2.
\]

It remains to identify the gain.  The usual within/between decomposition for
the fiber barycenters gives
\begin{align*}
 \sum_{j=1}^m p_j\|z_j\|_M^2
 &=\sum_{j\leq i}p_j\|z_j-a_i\|_M^2
   +\sum_{j>i}p_j\|z_j-b_i\|_M^2\\
 &\quad+P_i\|a_i\|_M^2+q_i\|b_i\|_M^2.
\end{align*}
Consequently,
\[
 V_i=T-\left(P_i\|a_i\|_M^2+q_i\|b_i\|_M^2\right).
\]
Since the barycenters are centered,
\[
 P_i a_i+q_i b_i=\sum_{j=1}^m p_jz_j
 =\mathbb E(X-\mu)=0.
\]
Using this identity, we obtain
\[
 P_i\|a_i\|_M^2+q_i\|b_i\|_M^2
 =P_iq_i\|a_i-b_i\|_M^2.
\]
Finally, $S_i=P_i a_i=-q_i b_i$, and hence
\[
 a_i-b_i=\frac{S_i}{P_i}+\frac{S_i}{q_i}
          =\frac{S_i}{P_iq_i}.
\]
Therefore
\[
 P_iq_i\|a_i-b_i\|_M^2
 =\frac{\|S_i\|_M^2}{P_iq_i},
\]
which proves all the identities in \eqref{eq:vector-gain}.
\end{proof}

We introduce now the vector form of the marginal signal.
Set
\begin{align}
  L_1^2&:=0,
  &L_i^2&:=\langle z_i-a_{i-1},z_i-a_i\rangle_M
       =\frac{P_{i-1}}{P_i}\|z_i-a_{i-1}\|_M^2,\quad 2\leq i\leq m,
       \label{eq:left-vector-radius}\\
  R_m^2&:=0,
  &R_i^2&:=\langle b_{i-1}-z_i,b_i-z_i\rangle_M
       =\frac{q_i}{q_{i-1}}\|b_i-z_i\|_M^2,\quad 1\leq i\leq m-1.
       \label{eq:right-vector-radius}
\end{align}
These quantities are nonnegative, including at the endpoints.  Define
\begin{equation}\label{eq:vector-signal}
  H_i:=R_i^2-L_i^2,
  \qquad 1\leq i\leq m.
\end{equation}

\begin{proposition}[Vector marginal identity]
\label{prop:vector-marginal-identity}
For $1\leq i\leq m$,
\begin{equation}\label{eq:vector-risk-increment}
  V_i-V_{i-1}=p_i(L_i^2-R_i^2)=-p_iH_i.
\end{equation}
Consequently, the polygon joining $(P_i,V_i)$ is convex if and only if
\begin{equation}\label{eq:weak-vector-monotonicity}
  H_1\geq H_2\geq\cdots\geq H_m.
\end{equation}
A sufficient, but stronger, condition is that $(L_i^2)$ is nondecreasing
and $(R_i^2)$ is nonincreasing.
\end{proposition}

\begin{proof}
When the $i$th fiber is moved to the left, the optimal left cost increases by
$p_iL_i^2$, while the optimal right cost decreases by $p_iR_i^2$.  This is the
same weighted sum-of-squares identity used in the scalar case and proves
\eqref{eq:vector-risk-increment}.  The slope on $[P_{i-1},P_i]$ is therefore
$-H_i$.  The last assertion follows immediately.
\end{proof}

Under \eqref{eq:weak-vector-monotonicity}, put
$\widehat H_0:=+\infty$, $\widehat H_i:=H_i$ for $1\leq i\leq m$, and
$\widehat H_{m+1}:=-\infty$.  A vertex $i$ minimizes the polygon precisely
when
\begin{equation}\label{eq:vector-crossing-test}
  \widehat H_i\geq0\geq \widehat H_{i+1}.
\end{equation}
The minimizing vertex is unique if both inequalities in
\eqref{eq:vector-crossing-test} are strict.
If $H_i=0$, the corresponding edge is flat and the two adjacent partitions
have the same value.  Notice that this can happen even though $M$ is strictly
positive definite.

\subsection{An anchored weak symmetry}
\label{subsec:anchored-vector-symmetry}

We need only a weak symmetry in the chosen direction, not invariance of the
whole distribution under $x\mapsto2\mu-x$.

\begin{definition}[Vector weak symmetry at the anchor]
\label{def:vector-weak-symmetry}
The direction $u$ is said to be weakly symmetric at the central anchor if,
for $Y=u^{\mathsf T}(X-\mu)$,
\begin{equation}\label{eq:vector-anchor}
  \mathbb P(Y<0)=\mathbb P(Y>0)=q>0,
  \qquad
  \mathbb E[(X-\mu)\mathbf 1_{\{Y=0\}}]=0.
\end{equation}
We write $r:=\mathbb P(Y=0)$, so that $2q+r=1$.
\end{definition}

Let
$A:=\mathbb E[(X-\mu)\mathbf 1_{\{Y<0\}}]$.  Centering and
\eqref{eq:vector-anchor} imply that the corresponding positive-side
moment is $-A$.  Moving the zero fiber, if it is present, has the common
left and right squared marginal radius
\begin{equation}\label{eq:central-common-radius}
  L(0)^2=R(0)^2
  =\frac{\|A\|_M^2}{q(q+r)},
  \qquad H(0)=0.
\end{equation}
When $r=0$, \eqref{eq:central-common-radius} is interpreted as a virtual
zero-mass index placed between the negative and positive fibers: the original
signals are left unchanged and the inserted quantities are defined by
\[
 L_*^2=R_*^2=\frac{\|A\|_M^2}{q^2},\qquad H_*=0.
\]

\begin{theorem}[Central cut under weak symmetry]
\label{thm:anchored-vector-cut}
Assume \eqref{eq:vector-anchor}, augment the ordered signal by the actual
or virtual zero fiber, and suppose that $H$ is nonincreasing on this
augmented order.  Then the central partition minimizes the fixed-direction
quadratic risk.  More precisely:
\begin{enumerate}
\item if $r=0$, $\{Y<0\}\mid\{Y>0\}$ is minimizing;
\item if $r>0$, both
$\{Y<0\}\mid\{Y\geq0\}$ and
$\{Y\leq0\}\mid\{Y>0\}$ are minimizing and have the same risk.
\end{enumerate}
If $H$ is strictly positive on every negative fiber and strictly negative
on every positive fiber, these are the only minimizing partitions.  Thus
there is one minimizing partition when $r=0$ and exactly two when $r>0$.
\end{theorem}

\begin{proof}
If the zero fiber is present, its conditional barycenter is zero by
\eqref{eq:vector-anchor}.  Substitution gives
\eqref{eq:central-common-radius}.  If $r=0$, the virtual insertion defines the
anchor signal without changing any of the original increments.  In both
cases the signal at the anchor is zero.  Since $H$ is nonincreasing, all the
increments in \eqref{eq:vector-risk-increment} before the anchor are
nonpositive and all those after the anchor are nonnegative.  When $r>0$, the
increment of the zero fiber is itself zero.  The strict assertion follows
from the strict signs away from the anchor.
\end{proof}

\section{From a fixed direction to all two-regime partitions}
\label{sec:global-halfspace}

We now remove the fixed direction.  The result below shows that, under
quadratic loss, it is enough to consider affine halfspaces.  The reason is the
usual nearest-centroid geometry.

For a Borel partition $A\mid A^c$ with $0<\mathbb P(X\in A)<1$, let
\begin{equation}\label{eq:global-binary-risk}
  \mathcal V(A):=\min_{a,b\in\mathbb R^d}
  \mathbb E\bigl[
     \|X-a\|_M^2\mathbf 1_{\{X\in A\}}
    +\|X-b\|_M^2\mathbf 1_{\{X\notin A\}}
  \bigr].
\end{equation}

\begin{proposition}[Scope of the halfspace search]
\label{prop:global-halfspace}
Assume that $X$ has finite, nonconstant support.  A global minimizer of
\eqref{eq:global-binary-risk} exists and may be chosen to be an affine
halfspace.  More precisely, let $A$ be a global minimizer and let
$a=\mathbb E[X\mid X\in A]$ and $b=\mathbb E[X\mid X\notin A]$.  Then
$a\neq b$ and, up to the law of $X$,
\begin{equation}\label{eq:nearest-centroid-sandwich}
 \bigl\{x:\langle M(b-a),x-(a+b)/2\rangle<0\bigr\}
 \subseteq A\subseteq
 \bigl\{x:\langle M(b-a),x-(a+b)/2\rangle\leq0\bigr\}.
\end{equation}
Thus $A$ is a halfspace whenever the bisecting hyperplane carries no
mass; in general, only the points on that hyperplane may be allocated
arbitrarily.  Assigning all bisector mass to one side, while retaining two
nonempty cells, produces a global halfspace minimizer.  Its normal is
$M(b-a)$.
\end{proposition}

\begin{proof}
Since the support is finite, there are only finitely many partitions and a
minimizer exists.  Moreover, nonconstancy gives a split with positive
between-regime gain.  Thus a global minimizer has smaller risk than the
one-regime representation, and in particular $a\neq b$.

Fix now the two representatives $a$ and $b$.  Every support point must be
assigned to a nearest representative; otherwise moving this point to the
other cell would strictly decrease the risk.  Since
\[
  \|x-a\|_M^2-\|x-b\|_M^2
  =2\langle M(b-a),x-(a+b)/2\rangle,
\]
we obtain \eqref{eq:nearest-centroid-sandwich} on the support.  Denote the
left-hand side of the last identity by $h(x)$.  Conditional centering gives
\[
 \E[h(X)\mid X\in A]=-\|a-b\|_M^2,
 \qquad
 \E[h(X)\mid X\notin A]=\|a-b\|_M^2.
\]
It follows that each cell has positive mass strictly on its own side of the
bisector.  We may therefore assign all the atoms on the bisector to either
side without emptying the other cell.  This does not change the loss for the
fixed representatives.  Reoptimizing cannot increase the loss, and global
minimality excludes a strict decrease.  The new partition is consequently a
global minimizer and is a halfspace up to the law of $X$.
\end{proof}

The proposition shows how the fixed-direction theory enters the global
problem: one may minimize first over the ordered cuts and then over the
direction $u$.  The nearest-centroid and within/between sum-of-squares
identities are standard in least-squares clustering and quantization; see, for
example,
\cite{du-faber-gunzburger,muresan-effros}.
Here we use them for a two-regime risk representation, and $H_i$ is interpreted
as a marginal risk signal.  In dimension one, maximizing the between-regime
variance is exactly Otsu's thresholding criterion \cite{otsu}.

We should stress that the halfspace conclusion is particular to quadratic
loss.  For a nonquadratic translation loss, the boundary is determined by the
equality of two translated loss functions and need not be affine.

\section{Multidimensional examples and counterexamples}
\label{sec:vector-examples}

We finish the vector discussion with two examples.  The first one shows that
an anchor or central symmetry does not replace the monotonicity of the
marginal signal.  The second one shows that the optimal direction need not be
unique.

\begin{example}[The anchor alone is insufficient]
\label{ex:central-symmetry-insufficient}
Let $M=I_2$ and let $X$ be uniform on
\[
  \{\pm(10,0),\ \pm(0,1),\ \pm(1,1)\}.
\]
The law is centrally symmetric and has mean zero.  For $u=(1,0)^{\mathsf T}$,
the negative and positive sides both have mass $1/3$, while the zero fiber
has mass $1/3$ and barycenter zero.  Thus \eqref{eq:vector-anchor} holds.
However, the augmented signal is
\[
  H=\left(120,-\frac{57}{2},0,\frac{57}{2},-120\right),
\]
which is not nonincreasing.  Each central partition has gain $61/4$,
whereas isolating $(10,0)$, or symmetrically $(-10,0)$, has gain $20$.
We can prove global optimality without listing all the partitions.  For a cell
$C$ of size $k\leq3$, symmetry gives
\[
 B(C)=\frac{\left\|\sum_{x\in C}x\right\|^2}{k(6-k)}.
\]
The largest squared subset sums for $k=1,2,3$ are respectively
$100,122,125$, and hence the largest possible gains are
$20,61/4,125/9$.  Thus the two extreme-point partitions are the global
minimizers.  In particular, neither central symmetry nor the anchor condition
can replace the monotonicity of the marginal signal.
\end{example}

\begin{example}[The globally optimal direction need not be unique]
\label{ex:direction-nonunique}
Let $M=I_2$ and let $X$ be uniform on the four vertices
$(\pm1,\pm1)$ of the square.  The total quadratic risk is $T=2$.
The vertical and horizontal splits each have two centroids at distance $2$,
so each has between-regime gain $1$ and residual risk $1$.
Every one-versus-three split has gain $2/3$, while a diagonal two-versus-two
split has gain zero.  Hence the horizontal and vertical partitions are the
two global minimizers, up to exchanging the cells.  Each fixed-direction
profile is simple, but the scalar ordering does not distinguish between the
two global normals.
\end{example}

\section{Conclusion}\label{sec:conclusion}

We have studied the quadratic two-regime problem first for finite laws on the
line and then for finitely supported random vectors.  In the scalar case, the
exact atom update gives the derivative of the profile in cumulative-mass
coordinates.  Its monotonicity is precisely the convexity of the profile.
Log-concavity gives useful sufficient conditions, while weak symmetry locates
the minimum once the correction at the boundary atoms is taken into account.

The continuous derivative has the same residual form, and the fine-grid
results explain the relation between the discrete and atomless settings.  In
several dimensions, conditional Pythagoras reduces a fixed-direction problem
to the ordered fiber barycenters, while nearest-centroid geometry produces a
global halfspace optimizer.  The examples show that support geometry,
monotonicity and the choice of direction remain essential.  The arguments also
make clear why the main conclusions are particular to quadratic loss.

\bibliographystyle{amsplain}
\bibliography{Discrete_7_arxiv}

@article{artzner-delbaen-eber-heath,
  author  = {Artzner, Philippe and Delbaen, Freddy and Eber, Jean-Marc and Heath, David},
  title   = {Coherent Measures of Risk},
  journal = {Mathematical Finance},
  volume  = {9},
  number  = {3},
  pages   = {203--228},
  year    = {1999},
  doi     = {10.1111/1467-9965.00068}
}

@misc{nistor-popescu-continuous,
  author        = {Nistor, Mihaela-Adriana and Popescu, Ionel},
  title         = {A Property of Log-Concave and Weakly-Symmetric Distributions for Two Step Approximations of Random Variables},
  year          = {2026},
  eprint        = {2603.12767},
  archiveprefix = {arXiv},
  howpublished  = {arXiv:2603.12767},
  url           = {https://arxiv.org/abs/2603.12767}
}

@article{otsu,
  author  = {Otsu, Nobuyuki},
  title   = {A Threshold Selection Method from Gray-Level Histograms},
  journal = {IEEE Transactions on Systems, Man, and Cybernetics},
  volume  = {9},
  number  = {1},
  pages   = {62--66},
  year    = {1979},
  doi     = {10.1109/TSMC.1979.4310076}
}

@article{trushkin1982,
  author  = {Trushkin, A. V.},
  title   = {Sufficient Conditions for Uniqueness of a Locally Optimal Quantizer for a Class of Convex Error Weighting Functions},
  journal = {IEEE Transactions on Information Theory},
  volume  = {28},
  number  = {2},
  pages   = {187--198},
  year    = {1982},
  doi     = {10.1109/TIT.1982.1056480}
}

@article{kieffer1983,
  author  = {Kieffer, John C.},
  title   = {Uniqueness of Locally Optimal Quantizer for Log-Concave Density and Convex Error Weighting Function},
  journal = {IEEE Transactions on Information Theory},
  volume  = {29},
  number  = {1},
  pages   = {42--47},
  year    = {1983},
  doi     = {10.1109/TIT.1983.1056622}
}

@article{muresan-effros,
  author  = {Muresan, Dan and Effros, Michelle},
  title   = {Quantization as Histogram Segmentation: Optimal Scalar Quantizer Design in Network Systems},
  journal = {IEEE Transactions on Information Theory},
  volume  = {54},
  number  = {1},
  pages   = {344--366},
  year    = {2008},
  doi     = {10.1109/TIT.2007.911170}
}

@article{johnson-goldschmidt,
  author  = {Johnson, Oliver and Goldschmidt, Christina},
  title   = {Preservation of Log-Concavity on Summation},
  journal = {ESAIM: Probability and Statistics},
  volume  = {10},
  pages   = {206--215},
  year    = {2006},
  doi     = {10.1051/ps:2006008}
}

@article{du-faber-gunzburger,
  author  = {Du, Qiang and Faber, Vance and Gunzburger, Max},
  title   = {Centroidal {Voronoi} Tessellations: Applications and Algorithms},
  journal = {SIAM Review},
  volume  = {41},
  number  = {4},
  pages   = {637--676},
  year    = {1999},
  doi     = {10.1137/S0036144599352836}
}

@article{duembgen-kolesnyk-wilke,
  author  = {D{\"u}mbgen, Lutz and Kolesnyk, Petro and Wilke, Ralf A.},
  title   = {Bi-Log-Concave Distribution Functions},
  journal = {Journal of Statistical Planning and Inference},
  volume  = {184},
  pages   = {1--17},
  year    = {2017},
  doi     = {10.1016/j.jspi.2016.10.005}
}

@article{prekopa1973,
  author  = {Pr{\'e}kopa, Andr{\'a}s},
  title   = {On Logarithmic Concave Measures and Functions},
  journal = {Acta Scientiarum Mathematicarum (Szeged)},
  volume  = {34},
  pages   = {335--343},
  year    = {1973}
}

@article{saumard-wellner,
  author  = {Saumard, Adrien and Wellner, Jon A.},
  title   = {Log-Concavity and Strong Log-Concavity: A Review},
  journal = {Statistics Surveys},
  volume  = {8},
  pages   = {45--114},
  year    = {2014},
  doi     = {10.1214/14-SS107}
}

@article{bagnoli-bergstrom,
  author  = {Bagnoli, Mark and Bergstrom, Ted},
  title   = {Log-Concave Probability and Its Applications},
  journal = {Economic Theory},
  volume  = {26},
  number  = {2},
  pages   = {445--469},
  year    = {2005},
  doi     = {10.1007/s00199-004-0514-4}
}

@article{frittelli-rosazza,
  author  = {Frittelli, Marco and Rosazza Gianin, Emanuela},
  title   = {Putting Order in Risk Measures},
  journal = {Journal of Banking \& Finance},
  volume  = {26},
  number  = {7},
  pages   = {1473--1486},
  year    = {2002},
  doi     = {10.1016/S0378-4266(02)00270-4}
}

@article{foellmer-weber,
  author  = {F{\"o}llmer, Hans and Weber, Stefan},
  title   = {The Axiomatic Approach to Risk Measures for Capital Determination},
  journal = {Annual Review of Financial Economics},
  volume  = {7},
  pages   = {301--337},
  year    = {2015},
  doi     = {10.1146/annurev-financial-111914-042031}
}

@article{rockafellar-uryasev-zabarankin-2006,
  author  = {Rockafellar, R. Tyrrell and Uryasev, Stan and Zabarankin, Michael},
  title   = {Generalized Deviations in Risk Analysis},
  journal = {Finance and Stochastics},
  volume  = {10},
  number  = {1},
  pages   = {51--74},
  year    = {2006},
  doi     = {10.1007/s00780-005-0165-8}
}

@article{fisher-1958,
  author  = {Fisher, Walter D.},
  title   = {On Grouping for Maximum Homogeneity},
  journal = {Journal of the American Statistical Association},
  volume  = {53},
  number  = {284},
  pages   = {789--798},
  year    = {1958},
  doi     = {10.1080/01621459.1958.10501479}
}

@article{lloyd-1982,
  author  = {Lloyd, Stuart P.},
  title   = {Least Squares Quantization in {PCM}},
  journal = {IEEE Transactions on Information Theory},
  volume  = {28},
  number  = {2},
  pages   = {129--137},
  year    = {1982},
  doi     = {10.1109/TIT.1982.1056489}
}

\end{document}